\documentclass[12pt,a4paper,oneside]{amsart}

\usepackage{amsfonts, amsmath, amssymb, amsthm, amscd, hyperref}
\usepackage{anysize}
\newtheorem{theorem}{Theorem}[section]
\newtheorem{lemma}[theorem]{Lemma}

\theoremstyle{definition}
\newtheorem{definition}[theorem]{Definition}

\theoremstyle{remark}
\newtheorem{remark}[theorem]{Remark}

\numberwithin{equation}{section}

\DeclareMathOperator{\ind}{ind}

\renewcommand{\epsilon}{\varepsilon}
\renewcommand{\phi}{\varphi}
\renewcommand{\kappa}{\varkappa}

\begin{document}

\title{Covering dimension using toric varieties}

\author{Roman~Karasev}
\thanks{Supported by the Dynasty Foundation, the President's of Russian Federation grant MD-352.2012.1, and the Russian government project 11.G34.31.0053.}
\email{r\_n\_karasev@mail.ru}
\address{Roman Karasev, Dept. of Mathematics, Moscow Institute of Physics and Technology, Institutskiy per. 9, Dolgoprudny, Russia 141700}
\address{Roman Karasev, Institute for Information Transmission Problems RAS, Bolshoy Karetny per. 19, Moscow, Russia 127994}
\address{Roman Karasev, Laboratory of Discrete and Computational Geometry, Yaroslavl' State University, Sovetskaya st. 14, Yaroslavl', Russia 150000}


\begin{abstract}
In this paper we deduce the Lebesgue and the Knaster--Kuratowski--Mazurkiewicz theorems on the covering dimension, as well as their certain generalizations, from some simple facts of toric geometry. This provides a new point of view on this circle of results.
\end{abstract}

\maketitle

\section{Introduction}

In the theory of covering dimension there are lemmas showing that the Euclidean space $\mathbb R^n$ has dimension at least $n$. One result is attributed to Lebesgue: If a unit cube $[0,1]^n$ is covered by compact sets so that no point is covered by more than $n$ of them, then one of the sets must intersect two opposite facets of the cube. This result shows that a covering of the cube (or the whole $\mathbb R^n$) with multiplicity at most $n$ cannot consist of arbitrarily small sets, and actually this conclusion is made quantitative, see also the review~\cite{guth2012} about this and other more sophisticated quantitative topological facts.

Another result is the Knaster--Kuratowski--Mazurkiewicz theorem~\cite{kkm1929} (usually referred to as the KKM theorem): If the simplex $S\subset \mathbb R^n$ is covered by a family of closed sets so that no point is covered by more than $n$ sets, then one of the sets intersects all the facets of $S$. This result also shows that a multiplicity $n$ covering cannot have arbitrarily small sets.

Different proofs of these lemmas are known. Probably, the most popular way is to subdivide the cube or the simplex into a triangulation and establish some combinatorial statement, like Sperner's lemma on coloring the vertices of a triangulation.

In this note we exploit a different approach and apply some elementary techniques of toric geometry (see~\cite{atiyah1983,ful1993} to get acquainted with the subject) to prove the KKM and the Lebesgue theorem and some their moderate generalizations.

\section{A KKM type theorem}

We start from proving the following strengthened KKM theorem, suggested by D\"om\"ot\"or P\'alv\"olgyi~\cite{domp2012}. This time we use the Borsuk--Ulam type technique developed in~\cite{kar2012}, which can be interpreted as a ``real toric'' approach. Later we will pass to the standard ``complex toric'' proof.

\begin{theorem}
\label{kkm}
Let $X_i\subset \Delta^n$, for $i=1, \ldots, N$, be closed subsets of the $n$-simplex such that every $X_i$ does not intersect some facet of $\Delta^n$. If no $k+1$ sets of $\{X_i\}$ have a common point then there exists a component of the complement $\Delta^n\setminus \bigcup_{i=1}^N X_i$ that intersects every $k$-face of $\Delta^n$.
\end{theorem}

\begin{remark}
For $k = n$, this theorem asserts that either some $n+1$ of the sets intersect, or they do not cover $\Delta^n$. This is the KKM theorem.
\end{remark}

\begin{proof}[Proof using the ``real toric'' approach]

Consider the map from the sphere $\pi : S^n \to \Delta^n$ defined by $\pi(x_0, \ldots, x_n) = (x_0^2, \ldots, x_n^2)$.

By the assumption any set $Y_i = \pi^{-1}(X_i)$ does not intersect the some hyperplane ${x_j = 0}$. Hence it consists of two parts, with positive and negative $x_j$ respectively, interchanged by the antipodal map $x\mapsto -x$. Therefore we can find an odd continuous function $f_i : S^n\to \mathbb R$ such that $f_i(Y_i) = \{+1, -1\}$ and $f_i(S^n\setminus Y_i) = (-1, 1)$.

Now consider the direct sum $f = f_1\oplus \dots \oplus f_N$, this is a map from $S^n$ to the cube $Q^N = [-1, 1]^N$. Put $Y = \bigcup_{i=1}^N Y_i$ and notice that $f(Y) \subseteq \partial Q^N$. From the condition that no $k+1$ of the sets $\{Y_i\}$ have nonempty intersection it follows that the image $f(Y)$ misses the $(N-1-k)$-dimensional skeleton $Q^N_{(N-1-k)}$ of the cube $Q^N$. Therefore the image $f(Y)$ can be equivariantly (with respect to the antipodal action of $\mathbb Z_2$) deformed to a $(k-1)$-dimensional subset of $\partial Q^N$, this subset can be described as the skeleton of the Poincar\'e dual partition of $\partial Q^N$.

From the standard properties of the $\mathbb Z_2$-index in terms of~\cite{mat2003} we observe the following: By the dimension bound we have $\ind f(Y) \le k-1$. Hence by the monotonicity of the index $\ind Y \le k-1$. By the subadditivity of the index it follows that $\ind S^n\setminus Y \ge n-k$. Moreover, there must exist a connected component $Z$ of the set $S^n\setminus Y$ with the same index $\ind Z \ge n-k$.

Now consider a $k$-face $F$ of $\Delta^n$, its preimage $\pi^{-1}(F)$ can be assumed to be defined by $x_{k+1}= x_{k+2} = \dots x_n = 0$, without loss of generality. Since $\ind Z\ge n-k$, the odd map $g : Z \to \mathbb R^{n-k}$ defined by the coordinates $(x_i)_{k+1}^n$ must meet the zero. Hence $Z\cap \pi^{-1}(F) \neq \emptyset$ and therefore $\pi(Z)\cap F\neq \emptyset$ for any $k$-face $F$. The set $\pi(Z)$ can be chosen to be the set required in the theorem.
\end{proof}

\section{The toric approach}

Let us prove Theorem~\ref{kkm} again using the approach of complex toric geometry, that is the genuine toric approach. In fact, in the previous section the map $S^n\to \Delta^n$ may be considered as the moment map of a ``real toric variety''; and now we are going to invoke the classical notion of a complex toric variety and its moment map.

The toric geometry studies symplectic manifolds $M^{2n}$ (or $n$-dimensional algebraic varieties), possibly with singularities of codimension at least $2$, as it usually happens with algebraic varieties. The symplectic (or K\"{a}hler) form $\omega\in \Omega^2(M)$ has the property that $\omega^n$ never vanishes and $\int_M \omega^n > 0$, so the cohomology class $[\omega]$ is defined and nontrivial in $H^2(M;\mathbb R)$ and $[\omega]^n\neq 0$. This is actually all we are going to use.

Let us make a definition:

\begin{definition}
A subset $Y\subseteq M$ is called \emph{inessential} if the class $[\omega]$ vanishes on $Y$, that is its image under the map $\iota^*$ is zero, where $\iota : Y\to M$ is the embedding.
\end{definition}

The main lemma that we use is the following:

\begin{lemma}
\label{covering-dim}
If $\{Y_i\}_{i=1}^N$ is a family of inessential open subsets of $M$ with covering multiplicity at most $k$ (no point of $M$ belongs to more than $k$ of them) then $[\omega]^k$ vanishes on their union $\bigcup_{i=1}^N Y_i$.
\end{lemma}

The proof of Lemma~\ref{covering-dim} is a combination of two steps. The first step is passing to a covering that is colored in $k$ colors in a compatible way, we borrow it from~\cite[Lemma~2.4]{palais1966}, although a weaker form of it can be shown by using the nerve theorem and passing to the barycentric subdivision, which is naturally regularly colored. Another, informal, point of view on this step is to assign the color $i$ to the set of those points that are covered exactly $i$ times by the sets of $\mathcal Y$, and then resolve the issue with not openness of the obtained covering. So this is the lemma:

\begin{lemma}[Palais, 1966]
\label{coloring}
An open covering $\mathcal Y$ of a paracompact space with multiplicity at most $k$ can be refined to a $k$-colorable open covering, that is a covering consisting of $k$ subfamilies $\mathcal Y_1, \ldots, \mathcal Y_k$ (colors) such that for any sets $Y_1, Y_2$ in the same color class $\mathcal Y_i$ the intersection $Y_1\cap Y_2$ is empty.
\end{lemma}

The next step will be a standard argument going back to Lusternik and Schnirelmann~\cite{ls1934}.

\begin{proof}[Proof of Lemma~\ref{covering-dim}]
By Lemma~\ref{coloring} we assume that $\mathcal Y = \bigcup_{i=1}^k \mathcal Y_i$ with every subfamily $\mathcal Y_i$ consisting of pairwise disjoint sets. Under taking the refinement the sets remain inessential, since this property is preserved under taking subsets.

The union $U_i = \bigcup \mathcal Y_i$ is a disjoint union and the open set $U_i$ is inessential, that is $[\omega]|_{U_i} = 0$ by definition. By the property of the cohomology multiplication $[\omega]^k|_U = 0$ for $U = \bigcup_{i=1}^k U_i = \bigcup \mathcal Y$, which contradicts the assumption.
\end{proof}

\begin{remark}
We could finish this proof without appealing to cohomology multiplication as follows: On every $U_i$ we find a differential form $\alpha_i$ such that $d\alpha_i = \omega$ on $U_i$. After taking smaller subsets $V_i\subseteq U_i$, still covering $U$, we may assume that the equality $d\alpha_i = \omega$ still holds on $V_i$ and $\alpha_i$ is extended to the whole $U = \bigcup \mathcal Y$. Now we observe that the product 
$$
(\omega - d\alpha_1)\wedge \dots \wedge (\omega - d\alpha_k)
$$
vanishes on the whole $U$, and after expanding we observe that $\omega^k = d\beta$ on $U$ for some $\beta\in \Omega^{2k-1}(U)$, which is nonsense.
\end{remark}

\begin{proof}[Toric proof of Theorem~\ref{kkm}]
Consider the moment map $\pi : \mathbb CP^n \to \Delta^n$, given in coordinates as 
$$
y_i = \frac{|z_i|^2}{|z_0|^2 + \dots + |z_n|^2}.
$$

Again, any set $Y_i = \pi^{-1}(X_i)$ does not intersect some hyperplane ${z_j = 0}$. These sets are closed, but by the standard technique we may pass to neighborhoods without increasing the covering multiplicity. Since there is a representative of $[\omega]$ concentrated near any given hyperplane in $\mathbb CP^n$ then $[\omega]$ vanishes on any $Y_i$ and those sets are inessential. 

By Lemma~\ref{covering-dim}, $[\omega]^k$ vanishes on $Y = \bigcup_i Y_i$. By the property of the cohomology multiplication, $[\omega]^{n-k}$ does not vanish on $\mathbb CP^n\setminus Y$ and there must exist a connected component $Z$ of the set $\mathbb CP^n\setminus Y$ with nonvanishing $[\omega]^{n-k}$.

As in the first proof, considering a $k$-face $F$ of $\Delta^n$ and its preimage $\pi^{-1}(F)$ we may assume that $\pi^{-1}(F) = \{z_{k+1}= x_{k+2} = \dots z_n = 0\}$. The class $[\omega]^{n-k}$ is Poincar\'e dual to the projective $k$-subspace $\pi^{-1}(F)$ and therefore has a representative concentrated in an arbitrarily small neighborhood of $\pi^{-1}(F)$. Hence the closed set $Z$ must intersect $\pi^{-1}(F)$ to have the class $[\omega]^{n-k}$ nontrivial; hence the set $\pi(Z)$ is what we need in the theorem.
\end{proof}

\section{Lebesgues's theorem}

We now see that the Lebesgue theorem has a very simple proof in the toric approach:

\begin{theorem}
\label{lebesgue}
Let the unit cube $Q^n$ be covered by a family of closed sets $X_i$ with covering multiplicity at most $n$. Then some $X_i$ touches two opposite facets of $Q^n$.
\end{theorem}

\begin{proof}
Consider the moment map $\pi : M\to Q^n$, where $M = (\mathbb CP^1)^n$ and the corresponding sets $Y_i = \pi^{-1} (X)$. 

Let $F_j^+$ and $F_j^-$ by the facets of $Q^n$ defined by $x_j = 1$ and $x_j=0$ respectively. If a given $Y_i$ does not touch both $F_j^+$ and $F_j^-$ at the same time, we may assume, without loss of generality for the following argument, that it does not touch $F_j^-$. 

It is known that the hyperplane divisor (Poincar\'e dual to the K\"ahler form) of the toric embedding of $M$ into $\mathbb CP^{2^n-1}$ is $\sum_{j=1}^n \left( \pi^{-1}(F_j^-) + \pi^{-1}(F_j^+) \right)$. Moreover, the latter divisor is equivalent to $D = \sum_{j=1}^n 2 \pi^{-1}(F_j^-)$; this follows from the general description of the divisor classes on toric varieties and is completely elementary for $M = (\mathbb CP^1)^n$. So we conclude that every $Y_i$ is inessential and the covering multiplicity at most $n$ is impossible by Lemma~\ref{covering-dim}.
\end{proof}

Similar to the generalization of the KKM theorem in Theorem~\ref{kkm}, it is possible to generalize the Lebesgue theorem as follows:

\begin{theorem}
\label{lebesgue-complement}
Let $\{X_i\}$ be a family of subsets of the unit cube $Q^n$ such that none of $X_i$ touches a pair of opposite facets of $Q^n$. If the covering multiplicity of $\{X_i\}$ is at most $k$ then there exists a connected component $Z$ of the complement $Q^n\setminus\bigcup_i X_i$ and a $k$-dimensional coordinate subspace $L\subseteq \mathbb R^n$ such that $Z$ intersects every $k$-face of $Q^n$ parallel to $L$.
\end{theorem}

\begin{proof}
Again, consider the moment map $\pi : M = (\mathbb CP^1)^n \to Q^n$ and the sets $Y_i = \pi^{-1}(X_i)$. By the assumption all of them are inessential. Hence, the cohomology class of $[\omega]^{n-k}$ is nonzero on the complement $M\setminus\bigcup_i Y_i$ and therefore is nonzero on a certain connected component $Z$ of this closed set.

Now we observe that $[\omega]^{n-k}$ is Poincar\'e dual to a positive multiple of the following $k$-dimensional subvariety of $M$. For every $I\subseteq\{1, \ldots, n\}$ with $|I|=k$ consider the coordinate subspace $\mathbb R^I$ and choose arbitrarily one face $F_I$ of $Q^n$ parallel to $F_I$. Then $[\omega]^{n-k}$ is Poincar\'e dual to $V_k = \sum_I \pi^{-1}(F_I)$, which can be checked by the explicit description of the cohomology of $M$ as $\mathbb Z[c_1, \ldots, c_n] / (c_1^2 = \dots = c_n^2 = 0)$.

Now we observe that $Z$ must intersect the $k$-subvariety $V_k$ for any choice of $F_I$. This precisely means that, for some fixed $I$, $Z$ must intersect all possible faces $\pi^{-1}(F_I)$.
\end{proof}

There is another version of the Lebesgue theorem (something like the Hex lemma~\cite{gale1979}) with a simple toric proof:

\begin{theorem}
\label{lebesgue-axes}
Let the unit cube $Q^n$ be covered by a family of closed sets $\{X_i\}_{i=1}^n$. Then some connected component of $X_i$ intersects both the corresponding opposite facets $F_i^+$ and $F_i^-$.
\end{theorem}

\begin{proof}
Assume the contrary. Then for every component $C$ of $\pi^{-1}(X_i)$ the class $[\omega_i]$, Poincar\'e dual to the divisor $\pi^{-1}(F_i^+) \sim \pi^{-1}(F_i^-)$, is zero on $C$. Therefore $[\omega_i]$ is zero when restricted to the whole $\pi^{-1}(X_i)$.

Then the product $[\omega_1]\cdot \dots \cdot [\omega_n]$ must be zero on the union $M = \bigcup_{i=1}^n \pi^{-1}(X_i)$, by the Lusternik--Schnirelmann argument. But from the correspondence between the cohomology product and the intersection of divisors this product is Poincar\'e dual the intersection $\bigcap_{i=1}^n \pi^{-1}(F_i^-)$, consisting of the single point. This is a contradiction.
\end{proof}

\section{Further results and problems}

Again, using the toric approach it is possible to reprove the result of~\cite{kar2012} (proved, as we now understand, with a ``real toric'' argument):

\begin{theorem}[The topological central point theorem]
Let $m=(d+1)(r-1)$, let $\Delta^m$ be the $m$-dimensional simplex, and let $W$ be a $d$-dimensional metric space. Suppose $f:\Delta^m\to W$ is a continuous map. Then 
$$
\bigcap_{\substack{F\subset\Delta^m\\\dim F = d(r-1)}} f(F)\neq\emptyset,
$$
where the intersection is taken over all faces of dimension $d(r-1)$.
\end{theorem}

\begin{proof}[Sketch of the proof]
For a fine enough open covering $\{U_i\}$ of $W$ of multiplicity at most $d+1$ we consider the corresponding covering of $\Delta^m$ by $X_i = f^{-1}(U_i)$, and the covering of $\mathbb CP^m$ by $Y_i = \pi^{-1}(X_i)$. Similar to the above proofs, from non-vanishing of $[\omega]^m$ over $\mathbb CP^m$ we conclude that for some $Y_i$ the class $[\omega]^{r-1}$ does not vanish over $Y_i$. Hence the corresponding $X_i$ must intersect all faces of $\Delta^m$ of appropriate dimension. Then a certain compactness argument establishes the same for a preimage of some $y\in W$.
\end{proof}

As another example we give a unified statement of the KKM and Lebesgue theorems:

\begin{theorem}
\label{kkm-lebesgue}
Let a simple polytope $P\subset \mathbb R^n$ be covered by a family of closed sets $X_i$ with covering multiplicity at most $n$. Then some $X_i$ touches at least $n+1$ facets of $P$.
\end{theorem}

\begin{proof}
For simple polytopes a moment map $\pi : M\to P$ can also be considered. In this general case $M$ may not be a smooth algebraic variety, but it is still an ``orbifold'', which means that it behaves like a smooth manifold with respect to the (co)homology with $\mathbb Q$ coefficients. This is sufficient to note that a cohomolgy class $[\omega]$ of an ample divisor on $P$ does not vanish in $n$-th power, and then apply the considerations from previous proofs. Put again $Y_i = \pi^{-1} (X)$. 

It is known that an ample divisor $H$ on $M$ can be obtained as $\pi^{-1}$ of a linear combination of facets $\sum c_F F$ with some positive coefficients. The divisors linearly equivalent to zero correspond to sums of the form $\sum a_F(v) \pi^{-1}(F)$, where the coefficients $a_F$ correspond to the respective  fluxes of a constant vector field $v$ through facets of $P$.

Now we observe that a simple polytope admits a perturbation of its facet hyperplanes and remains homeomorphic to itself under such perturbations. So we assume that no $n$ normals to facets of $P$ are linearly dependent. Therefore if any covering set $X_i$ touches at most $n$ facets $F_i, \ldots, F_m$ of $P$, then we can find a divisor $H'$ linearly equivalent to $H$ and having zero coefficients at these $F_i$'s. Indeed, it is possible to select a vector $v$ that has prescribed fluxes through these $n$ facets. So assuming the contrary, we obtain that all sets $Y_i$ do not intersect an appropriately chosen $H'$, are inessential, and cannot cover $M$ with multiplicity at most $n$.
\end{proof}

These inspiring connection between the toric geometry and the combinatorics of coverings and colorings raises several questions. Recall the Alexandrov width (waist) theorem (see~\cite[Theorem~6.2]{kar2012} and also~\cite{sit1958}): If a convex body $K\subset \mathbb R^n$ is covered by a family $\{X_i\}_{i=1}^N$ of closed sets with multiplicity at most $n$ then some $X_i$ cannot be put into a smaller homothetic copy of $K$. In fact, Theorem~\ref{kkm-lebesgue} is a particular case of this result. 

Trying to apply the technique of Theorem~\ref{kkm-lebesgue} to the Alexandrov width theorem, we encounter some problems. It would be nice if all subsets of $K$ that can be covered by a smaller copy of $K$ were inessential, but it is not true. For example, when $K$ is a polytope, an inessential set $X\subset K$ cannot contain an edge of $K$. This is because an edge $E\subset K$ corresponds to a curve $\pi^{-1}(E)$, and $\int_{\pi^{-1}(E)} \omega >0$. So if $K$ has an edge, which if not an \emph{affine diameter} (maximal linear section in given direction), then this edge can be covered by a smaller copy of $K$ and \emph{is} essential.

So it seems that the Alexandrov width theorem is unlikely to be proved in the same way as Theorems~\ref{kkm} and \ref{lebesgue}.

Another question is how to use the toric approach in the problems of estimating the maximal $(n-m)$-dimensional measures of subsets that cover the cube $Q^n$ with multiplicity at most $m$, similar to results in~\cite{kar2011cube,matd2011}.

\end{document}